\documentclass[letterpaper, 10 pt, conference]{ieeeconf}  

\IEEEoverridecommandlockouts                              

\usepackage{amssymb,latexsym,amsmath}
\usepackage{mathrsfs}
\usepackage{epsfig}
\usepackage{caption}

\usepackage{dsfont}
\usepackage{color}
\usepackage{epsfig}
\usepackage{caption}
\usepackage[normalem]{ulem}

\usepackage{lipsum}
\usepackage{epstopdf}
\usepackage{algorithmic}
\ifpdf
  \DeclareGraphicsExtensions{.eps,.pdf,.png,.jpg}
\else
  \DeclareGraphicsExtensions{.eps}
\fi
\newtheorem{example}{Example}[section]
       \newtheorem{theorem}{\bf{Theorem}}[section]
       
       \newtheorem{corollary}{\bf{Corollary}}[section]
       \newtheorem{lemma}{\bf{Lemma}}[section]
       \newtheorem{assumption}{\bf{Assumption}}[section]
\newtheorem{remark}{\bf{Remark}}[section]
        \newtheorem{definition}{\bf{Definition}}[section]

\usepackage{color}%

\def\sY{{\mathbb Y}}
\def\sU{{\mathbb U}}

\def\P{{\mathcal P}}

\allowdisplaybreaks

\title{Partially Observed Optimal Stochastic Control: Regularity, Optimality, Approximations, and Learning}

\author{Ali Devran Kara$^1$ and Serdar Y\"uksel$^2$
\thanks{$^{1}$Department of Mathematics at Florida State University
        {\tt\small akara@fsu.edu}}
\thanks{$^{2}$Department of Mathematics and Statistics at Queen's University, Kingston ON, Canada.
        {\tt\small yuksel@queensu.ca}. This research was
    partially supported by the Natural Sciences and Engineering
    Research Council of Canada (NSERC).}
}

\begin{document}

\maketitle
\thispagestyle{empty}
\pagestyle{empty}


\begin{abstract}
In this review/tutorial article, we present recent progress on optimal control of partially observed Markov Decision Processes (POMDPs). We first present regularity and continuity conditions for POMDPs and their belief-MDP reductions, where these constitute weak Feller and Wasserstein regularity and controlled filter stability. These are then utilized to arrive at existence results on optimal policies for both discounted and average cost problems, and regularity of value functions. Then, we study rigorous approximation results involving quantization based finite model approximations as well as finite window approximations under controlled filter stability. Finally, we present several recent reinforcement learning theoretic results which rigorously establish convergence to near optimality under both criteria. 
\end{abstract}


\section{Partially Observed Markov Decision Processes: Introduction and Preliminaries}


%
%
%

Partially observed Markov Decision processes (POMDPs) present challenging mathematical problems with significant applied relevance.  




Consider a stochastic process $\{X_k, k \in \mathbb{Z}_+\}$, where each element $X_k$ takes values in some standard Borel space $\mathbb{X}$, with dynamics described by
\begin{eqnarray}\label{updateEq}
 X_{k+1} &=& F(X_k, U_k, W_k) \label{updateEq1} \\
Y_{k} &=& G(X_k, V_k) \label{updateEq2} 
\end{eqnarray}
where $Y_k$ is an $\mathbb{Y}$-valued measurement sequence; we take $\mathbb{Y}$ also to be some standard Borel space. Suppose further that $X_0 \sim \mu$ for some $\mu\in\P(\mathds{X})$, where $\mathcal{P}(\mathds{X})$ represents the set of all probability measures on $\mathds{X}$. Here, $W_k, V_{k}$ are mutually independent i.i.d. noise processes. This system is subjected to a control/decision process where the control/decision at time $n$, $U_k$, incurs a cost $c(X_k,U_k)$. The decision maker only has access to the measurement process $Y_k$ and $U_k$ causally: An {\em admissible policy} $\gamma$ is a
sequence of control/decision functions $\{\gamma_k,\, k\in \mathbb{Z}_+\}$ such
that $\gamma_k$ is measurable with respect to the $\sigma$-algebra
generated by the information variables
\[
I_k=\{Y_{[0,k]},U_{[0,k-1]}\}, \quad k \in \mathbb{N}, \quad
  \quad I_0=\{Y_0\}.
\]
so that
\begin{equation}
\label{eq_control}
U_k=\gamma_k(I_k),\quad k\in \mathbb{Z}_+
\end{equation}
are the $\mathbb{U}$-valued control/decision
actions and  we use the notation
\[
Y_{[0,k]} = \{Y_s,\, 0 \leq s \leq k \}, \quad U_{[0,k-1]} =
  \{U_s, \, 0 \leq s \leq k-1 \}.
\]
We define $\Gamma$ to be the set of all such (strong-sense) admissible policies. We emphasize the implicit assumption that the control policy also depends on the prior probability measure $\mu$. 

We assume that all of the random variables are defined on a common probability space $(\Omega, {\cal F}, P)$ given the initial distribution on the state, and a policy,  on the infinite product space consistent with finite dimensional distributions, by the Ionescu Tulcea Theorem \cite{HernandezLermaMCP}. We will sometimes write the probability measure on this space as $P_\mu^\gamma$ to emphasize the policy $\gamma$ and the initialization $\mu$. We note that (\ref{updateEq1})-(\ref{updateEq2}) can also, equivalently (via stochastic realization results \cite[Lemma~1.2]{gihman2012controlled} \cite[Lemma~3.1]{BorkarRealization}, \cite[Lemma F]{aumann1961mixed}), be represented with transition kernels: the state transition kernel is denoted with $\mathcal{T}$ so that for Borel $B \subset \mathbb{X}$ \[{\mathcal{T}}(B|x,u) := P(X_1 \in B | X_0=x,U_0=u), \quad.\] We will denote the measurement kernel with $Q$ so that for Borel $B \subset \mathbb{Y}$: \[Q(B|x) := P(Y_0 \in B | X_0=x).\]

For (\ref{updateEq})-(\ref{updateEq2}), we are interested in minimizing either the average-cost optimization criterion
\begin{eqnarray}\label{expCost}
J_{\infty}(\mu,\gamma) := \limsup_{N \to \infty} {1 \over N} E^{\gamma}_{\mu}[\sum_{k=0}^{N-1} c(X_k, U_k)] 
\end{eqnarray}
or the discounted cost criterion (for some $\beta \in (0,1)$
\begin{eqnarray}\label{expDiscCost}
J_{\beta}(\mu,\gamma) :=E^{\gamma}_{\mu}[\sum_{k=0}^{\infty} \beta^k c(X_k, U_k)] 
\end{eqnarray}
over all admissible control policies $\gamma = \{\gamma_0, \gamma_1, \cdots,\} \in \Gamma$ with $X_0 \sim \mu$. With ${\cal P}(\mathbb{U})$ denoting the set of probability measures on $\mathbb{U}$ endowed with the weak convergence topology, we will also, when needed, allow for independent randomizations so that $\gamma_k(I_k)$ is ${\cal P}(\mathbb{U})$-valued for each realization of $I_k$. Here $c: \mathbb{X} \times \mathbb{U} \to \mathbb{R}_+$ is the cost function. 

One may also consider the control-free case where the system equation (\ref{updateEq1}) does not have control dependence; in this case only a decision is to be made at every time stage; $U$ is present only in the cost expression in (\ref{expCost}). This important special case has been studied extensively in the theory of non-linear filtering.

\subsection{Literature review and preliminaries}
In the following, we present a brief literature review on optimal control of POMDPs, before presenting the main results of the article. 

\noindent{\bf POMDPs, separated policies and belief-MDPs.} It is well-known that any POMDP can be reduced to a (completely observable) MDP \cite{Yus76}, \cite{Rhe74}, whose states are the posterior state probabilities, or beliefs, of the observer; that is, the state at time $k$ is
\begin{align}
\pi_k(\,\cdot\,) := P\{X_{k} \in \,\cdot\, | Y_0,\ldots,Y_k, U_0, \ldots, U_{k-1}\} \in {\cal P}(\mathbb{X}). \nonumber
\end{align}
We call this equivalent MDP the belief-MDP\index{Belief-MDP}. The belief-MDP has state space ${\cal P}(\mathbb{X})$ and action space $\sU$. Here, ${\cal P}(\mathbb{X})$ is equipped with the Borel $\sigma$-algebra generated by the topology of weak convergence \cite{Bil99}. Since $\mathbb{X}$ is a Borel space, ${\cal P}(\mathbb{X})$ is metrizable with the Prokhorov metric which makes ${\cal P}(\mathbb{X})$ into a Borel space \cite{Par67}. The transition probability $\eta$ of the belief-MDP can be constructed as follows (see also \cite{Her89}). If we define the measurable function 
\[F(\pi,a,y) := Pr\{X_{k+1} \in \,\cdot\, | \pi_k = \pi, U_k = u, Y_{k+1} = y\}\]
 from ${\cal P}(\mathbb{X})\times\mathbb{U}\times\sY$ to ${\cal P}(\mathbb{X})$ and the stochastic kernel $H(\,\cdot\, | \pi,u) := Pr\{Y_{k+1} \in \,\cdot\, | \pi_k = \pi, U_k = u\}$ on $\sY$ given ${\cal P}(\mathbb{X}) \times \sU$, then $\eta$ can be written as
\begin{align}
\eta(\,\cdot\,|\pi,u) = \int_{\sY} 1_{\{F(\pi,u,y) \in \,\cdot\,\}} H(dy|\pi,u). \label{kernelFilter}
\end{align}
The one-stage cost function $c$ of the belief-MDP is given by
\begin{align}
\tilde{c}(\pi,u) := \int_{\mathbb{X}} c(x,u) \pi(dx). \label{weak:eq8}
\end{align}
With cost function $c(x,u)$ is continuous and bounded on $\mathbb{X} \times \mathbb{U}$, an application of the generalized dominated convergence theorem \cite[Theorem 3.5]{Lan81} \cite[Theorem 3.5]{serfozo1982convergence}, we have that $\tilde{c}(\pi,u)=E^{\pi}[c(x,u)]:=\int \pi(dx) c(x,u): {\cal P}(\mathbb{X}) \times \mathbb{U} \to \mathbb{R}$ is also continuous and bounded, and thus Borel measurable.

In particular, the belief-MDP is a (fully observed) Markov decision process with the components $({\cal P}(\mathbb{X}),\sU,\eta,\tilde{c})$. 

For finite horizon problems and a large class of infinite horizon discounted cost problems, it is a standard result that an optimal control policy will use the belief $\pi_k$ as a sufficient statistic for optimal policies (see \cite{Yus76,Rhe74,Blackwell2}).

{\bf Approximations and Learning for POMDPs.} Studies on POMDPs had primarily been algorithmic and numerical, with rigorous studies applicable to a particular set of problems until recently. In particular, the regularity properties of POMDPs as pioneered in \cite{CrDo02,FeKaZg14} and later generalized to further conditions and criteria in \cite{KSYWeakFellerSysCont,feinberg2022markov,feinberg2023equivalent,kara2020near,demirci2023geometric} have paved the way for rigorous and explicit performance bounds. For approximate optimality on POMDPs, we refer the reader to the detailed review in \cite{SYLTAC2017POMDP} for discounted cost problems and \cite{demirci2023average} for the average cost problem. 

If the agent does not know the underlying dynamics of the observations (transitions and/or channel), then the learning of the solutions to the optimal control problem from observed data is necessary. However, learning for POMDPs has been a challenging problem. Various approaches and studies are available in the literature starting with \cite{singh1994learning}, see e.g. \cite{even2005reinforcement,jin2020sample,xiong2022sublinear,kwon2021rl,azizzadenesheli2016reinforcement} for some of the learning approaches for POMDPs.  In particular, we also cite the recent comprehensive studies \cite{chandak2022reinforcement} and \cite{dong2022simple} which study learning in non-Markov environments. \cite{subramanian2022approximate,seyedsalehi2023approximate} present a general framework on approximation states and their induced optimality and near optimality properties under several uniformity bounds. Some of our explicit analysis here can also be seen in view of these bounds. We also refer the reader to the second tutorial paper \cite{CDCTutorial2024SInhaMahajan} for complementary approaches to the planning and learning problem. 

Our technical approach on learning builds on the mathematical analysis developed in \cite{kara2021convergence} and generalized in \cite{karayukselNonMarkovian} (see also \cite{KSYContQLearning,kara2020near}). We note that more general, non-uniform, bounds are also considered in \cite{KSYContQLearning,kara2020near}.

\noindent{\bf Control-free setup.} For the special case without control, the belief process is known as the (non-linear) filter process, and by the discussion above, this itself is a Markov process. For our paper, this setup will be useful to study the convergence and uniqueness properties involving invariant probability measures for the filter process: The stability properties of such processes has been studied, where the existence of an invariant probability measure for the belief process, as well as the uniqueness of such a measure (i.e., the unique ergodicity property) has been investigated under various conditions, see. e.g. \cite{budhiraja2002invariant} and \cite{chigansky2010complete} which provide a comprehensive discussion on both the ergodicity of the filter process as well as filter stability. \cite[Theorem 2]{DiMasiStettner2005ergodicity} and \cite[Prop 2.1]{van2009uniformSPA} assume that the hidden state process is ergodic and the filter is stable (almost surely or in expectation under total variation); these papers crucially embed the stationary state in the joint process $(x_k,\pi_k)$ and note that when $x_k$ is stationary, the Markov chain defined by this process admits an invariant probability measure. See also \cite{Kaijser,kochman2006simple,chigansky2010complete,kaijser2011markov,Szarek} for further filter stability and unique ergodicity results and a recent review in \cite{anotherLookPOMDPs}.

\subsection{Convergence Notions for Probability Measures}
For the analysis of the technical results, we will use different convergence and distance notions for probability measures. Two important notions of convergences for sequences of probability measures are weak convergence, and convergence under total variation. For some $N\in\mathbb{N}$ a sequence $\{\mu_n,n\in\mathbb{N}\}$ in $\mathcal{P}(\mathbb{X})$ is said to converge to $\mu\in\mathcal{P}(\mathbb{X})$ \emph{weakly} if $\int_{\mathbb{X}}f(x)\mu_n(dx) \to \int_{\mathbb{X}}f(x)\mu(dx)$ for every continuous and bounded $c:\mathbb{X} \to \mathbb{R}$.
One important property of weak convergence is that the space of probability measures on a complete, separable, and metric (Polish) space endowed with the topology of weak convergence is itself complete, separable, and metric \cite{Par67}. One such metric is the bounded Lipschitz metric  \cite[p.109]{villani2008optimal}, which is defined for $\mu,\nu \in {\cal P}(\mathbb{X})$ as 
\begin{equation}\label{BLmetric}
\rho_{BL}(\mu,\nu):=\sup_{\|f\|_{BL}\leq1} | \int f d\mu - \int f d\nu | 
\end{equation}
where \[ \|f\|_{BL}:=\|f\|_\infty+\sup_{x\neq y}\frac{|f(x)-f(y)|}{d(x,y)} \]
and $\|f\|_\infty=\sup_{x\in\mathbb{X}}|f(x)|$.

We next introduce the first order Wasserstein metric. The \emph{Wasserstein metric} of order 1 for two measures $\mu,\nu\in\mathcal{P}(\mathbb{X})$ is defined as
\begin{align*}
  W_1(\mu,\nu) =\inf_{\eta \in \mathcal{H}(\mu,\nu)} \int_{\mathbb{X}\times\mathbb{X}} \eta(dx,dy) |x-y|,
\end{align*}
where $\mathcal{H}(\mu,\nu)$ denotes the set of probability measures on $\mathbb{X}\times\mathbb{X}$ with first marginal $\mu$ and second marginal $\nu$. Furthermore, using the dual representation of the first order Wasserstein metric, we equivalently have
\begin{align*}
W_1(\mu,\nu)=\sup_{Lip(f)\leq 1}\left|\int f(x)\mu(dx)-\int f(x)\nu(dx)\right|
\end{align*}
where $Lip(f)$ is the minimal Lipschitz constant of $f$.

A sequence $\{\mu_n\}$ is said to converge in $W_1$ to $\mu \in \mathcal{P}(\mathds{X})$ if $W_1(\mu_n,\mu)  \to 0$. For compact $\mathbb{X}$, the Wasserstein distance of order $1$ metrizes the weak topology on the set of probability measures on $\mathbb{X}$ (see \cite[Theorem 6.9]{villani2008optimal}). For non-compact $\mathbb{X}$ convergence in the $W_1$ metric implies weak convergence (in particular this metric bounds from above the Bounded-Lipschitz metric \cite[p.109]{villani2008optimal}, which metrizes the weak convergence).

  For probability measures $\mu,\nu \in \mathcal{P}(\mathbb{X})$, the \emph{total variation} metric is given by
  \begin{align*}
    \|\mu-\nu\|_{TV}&=\sup_{f:\|f\|_\infty \leq 1}\left|\int f(x)\mu(dx)-\int f(x)\nu(dx)\right|,
  \end{align*}
  \noindent where the supremum is taken over all measurable real $f$ such that $\|f\|_\infty=\sup_{x\in\mathbb{X}}|f(x)|\leq 1$. A sequence $\mu_n$ is said to converge in total variation to $\mu \in \mathcal{P}(\mathbb{X})$ if $\|\mu_n-\mu\|_{TV}\to 0$.

\section{Regularity Results: Weak Continuity, Wasserstein Continuity, Wasserstein Contraction and Filter Stability}

\subsection{Weak Feller Continuity of the Belief-MDP}

Building on \cite{KSYWeakFellerSysCont} and \cite{FeKaZg14}, this section establishes the weak Feller property of the filter process; that is, the weak Feller property of the kernel defined in (\ref{kernelFilter}) under two different sets of assumptions.

\begin{assumption}\label{TV_channel}
\begin{itemize}
\item[(i)] The transition probability $\mathcal{T}(\cdot|x,u)$ is weakly continuous (weak Feller) in $(x,u)$, i.e., for any $(x_n,u_n)\to (x,u)$, $\mathcal{T}(\cdot|x_n,u_n)\to \mathcal{T}(\cdot|x,u)$ weakly.
\item[(ii)] The observation channel $Q(\cdot|x,u)$ is continuous in total variation, i.e., for any $(x_n,u_n) \to (x,u)$, $Q(\cdot|x_n,u_n) \rightarrow Q(\cdot|x,u)$ in total variation.
\end{itemize}
\end{assumption}

\begin{assumption}\label{TV_kernel}
\begin{itemize}
\item[(i)] The transition probability $\mathcal{T}(\cdot|x,u)$ is continuous in total variation in $(x,u)$, i.e., for any $(x_n,u_n)\to (x,u)$, $\mathcal{T}(\cdot|x_n,u_n) \to \mathcal{T}(\cdot|x,u)$ in total variation.
\item[(ii)] The observation channel $Q(\cdot|x)$ is independent of the control variable.
\end{itemize}
\end{assumption}

\begin{theorem} \cite{FeKaZg14} \label{TV_channel_thm}
Under Assumption \ref{TV_channel}, the transition probability $\eta(\cdot|z,u)$ of the filter process is weakly continuous in $(z,u)$.
\end{theorem}

\begin{theorem} \cite{KSYWeakFellerSysCont} \label{TV_kernel_thm}
Under Assumption \ref{TV_kernel}, the transition probability $\eta(\cdot|z,u)$ of the filter process is weakly continuous in $(z,u)$.
\end{theorem}

%
%

As examples, taken from \cite{KSYWeakFellerSysCont}, suppose that the system dynamics and the observation channel are represented as follows:
\begin{align*}
x_{t+1} &= H(x_t,u_t,w_t),\\
y_t &= G(x_t,u_{t-1},v_t),
\end{align*}
where $w_t$ and $v_t$ are i.i.d. noise processes. 

\begin{itemize}
\item[(i)]
Suppose that $H(x,u,w)$ is a continuous function in $x$ and $u$. Then, the corresponding transition kernel is weakly continuous. 
\item[(ii)]\label{ornii}
Suppose that $G(x,u,v)=g(x,u)+v$, where $g$ is a continuous function and $V_t$ admits a continuous density function $\varphi$ with respect to some reference measure $\nu$. Then, the channel is continuous in total variation. 
\item[(iii)]
Suppose that we have $H(x,u,w)=h(x,u)+w$, where $f$ is continuous and $w_t$ admits a continuous density function $\varphi$ with respect to some reference measure $\nu$. Then, the transition probability is continuous in total variation. 
\end{itemize}

\subsection{Wasserstein Continuity and Contraction Properties of the Belief-MDP}\label{wass_cont}


Recently, \cite{kara2020near} presented the following regularity results for controlled filter processes. Let us first recall the following:

\begin{definition}\cite[Equation 1.16]{dobrushin1956central}[Dobrushin coefficient]
For a kernel operator $K:S_{1} \to \mathcal{P}(S_{2})$ (that is a regular conditional probability from $S_1$ to $S_2$) for standard Borel spaces $S_1, S_2$, we define the Dobrushin coefficient as:
\begin{align}
\delta(K)&=\inf\sum_{i=1}^{n}\min(K(x,A_{i}),K(y,A_{i}))\label{Dob_def}
\end{align}
where the infimum is over all $x,y \in S_{1}$ and all partitions $\{A_{i}\}_{i=1}^{n}$ of $S_{2}$.
\end{definition}



\begin{assumption}\label{main_assumption}
\noindent
\begin{enumerate}
\item \label{compactness}
$(\mathbb{X}, d)$ is a bounded compact metric space 
with diameter $D$ (where $D=\sup_{x,y \in \mathbb{X}} d(x,y)$).
\item \label{totalvar}
The transition probability $\mathcal{T}(\cdot \mid x, u)$ is 
continuous in total variation in $(x, u)$, i.e., 
for any $\left(x_n, u_n\right) \rightarrow(x, u), 
\mathcal{T}\left(\cdot \mid x_n, u_n\right) \rightarrow 
\mathcal{T}(\cdot \mid x, u)$ in total variation.
\item \label{regularity}
There exists 
$\alpha \in R^{+}$such that 
$$
\left\|\mathcal{T}(\cdot \mid x, u)-\mathcal{T}\left(\cdot \mid x^{\prime}, u\right)\right\|_{T V} \leq \alpha d(x, x^{\prime})
$$
for every $x,x' \in \mathbb{X}$, $u \in \mathbb{U}$.
\item \label{CostLipschitz}
There exists $K_1 \in \mathbb{R}^+$ such that
\[|c(x,u) - c(x',u)| \leq K_1 d(x,x').\]
for every $x,x' \in \mathbb{X}$, $u \in \mathbb{U}$.
\item The cost function $c$ is bounded and continuous.
\end{enumerate}
\end{assumption}

\begin{theorem}\label{ergodicity}\cite{demirci2023average}
    Assume that $\mathbb{X}$ and $\mathbb{Y}$ are Polish spaces. 
    If Assumptions \ref{main_assumption}-\ref{compactness},\ref{regularity} are 
    fulfilled, then we have
    $$
    W_{1}\left(\eta(\cdot \mid z_0, u), \eta\left(\cdot \mid z_0^{\prime},u\right)\right) 
    \leq K_2 W_{1}\left(z_0, z_0^{\prime}\right),$$
with    $K_2:=\frac{\alpha D (3-2\delta(Q))}{2}$
    for all $z_0,z_0' \in \cal{P}(\mathbb{X})$, $u \in \mathbb{U}$.
\end{theorem}


\begin{remark}A recent paper \cite{anotherLookPOMDPs} has presented an alternative approach, without belief-separation, and has arrived further conditions for the existence of optimal policies for discounted and average cost problems as well as the unique ergodicity property for both controlled and control-free setups. Such an approach leads to complementary conditions on the weak Feller property on the state, which considers the entire past as the state endowed with the product topology.
\end{remark}

\subsection{Filter Stability}\label{filter_stability_sec}

The filter stability problem refers to the correction of an incorrectly initialized non-linear filter for a partially observed stochastic dynamical system (controlled or control-free) with increasing measurements. As we will see, this property has significant implications on robustness as well as near optimality of sliding finite window policies with explicit approximation bounds, to be presented in the paper.

Let us describe this property more explicitly: Given a prior $\mu\in \mathcal{P}(\mathbb{X})$ and a policy $\gamma \in {\Gamma}$ we can define the filter and predictor for a POMDP using the (strategic) measure {$P^{\mu,\gamma}$}.
\begin{definition}
\begin{itemize}
\item[(i)] We define the one step predictor process as the sequence of conditional probability measures
\begin{align*}
\pi_{n-}^{\mu,\gamma}(\cdot)=P^{\mu,\gamma}(X_{n} \in \cdot|Y_{[0,n-1]},U_{[0,n-1]}) ~~~~n\in \mathbb{N}
\end{align*}
\item[(ii)]We define the filter process as the sequence of conditional probability measures
\begin{align}\label{filterDefn}
\pi_{k}^{\mu,\gamma}(\cdot)=P^{\mu,\gamma}(X_{k} \in \cdot|Y_{[0,k]},U_{[0,k-1]}) , \quad n\in\mathbb{Z}_+
\end{align}
\end{itemize}
\end{definition}
\begin{remark}\label{remark:filter_control}
Recall that the $U_{[0,k-1]}$ are all functions of the $Y_{[0,k-1]}$, so conditioning on the control actions is not necessary in the above definitions. Yet this conditional probability would be {\it policy dependent}; if we condition on the past actions, this conditioning would be {\it policy-independent}. 
\end{remark} 
Say a prior $\mu \in \mathcal{P}(\mathbb{X})$ and a policy $\gamma \in \Gamma$ are chosen, an observer sees measurements $Y_{[0,\infty)}$ generated via the strategic measure $P^{\mu,\gamma}$. The observer is aware that the policy applied is $\gamma$, but incorrectly thinks the prior is $\nu\neq \mu$. The observer will then compute the incorrectly initialized filter $\pi_{k}^{\nu,\gamma}$ while the true filter is $\pi_{k}^{\mu,\gamma}$. The filter stability problem is concerned with the merging of $\pi_{k}^{\nu,\gamma}$ and $\pi_{k}^{\mu,\gamma}$ as $k$ goes to infinity.

In the literature, there are a number of merging notions when one considers stability which we enumerate here. Let $C_{b}(\mathbb{X})$ represent the set of continuous and bounded functions from $\mathbb{X} \to \mathbb{R}$. We define here the different notions of stability for the filter.

\begin{definition}\label{filterStabilityAsymptotic}
\begin{itemize}
\item[(i)] A filter process is said to be stable in the sense of weak merging with respect to a policy $\gamma$, $P^{\mu,\gamma}$ almost surely (a.s.) if there exists a set of measurement sequences $A \subset \mathcal{Y}^{\mathbb{Z}_{+}}$ with $P^{\mu,\gamma}$ probability 1 such that for any sequence in $A$; for any $f \in C_{b}(\mathcal{X})$ and any prior $\nu$ with $\mu \ll \nu$ (i.e., for all Borel $B$ $\nu(B) = 0 \implies \mu(B)=0$) we have
$
\lim_{n \to \infty} \left|\int f d\pi_{n}^{\mu,\gamma}-\int f d\pi_{n}^{\nu,\gamma}\right|=0
$.

\item[(ii)]  A filter process is said to be stable in the sense of total variation in expectation with respect to a policy $\gamma$ if for any measure $\nu$ with $\mu \ll \nu$ we have
$
\lim_{n \to \infty} E^{\mu,\gamma}[\|\pi_{n}^{\mu,\gamma}-\pi_{n}^{\nu,\gamma}\|_{TV}]=0
$.

\item[(iii)]  A filter process is said to be stable in the sense of total variation with respect to a policy $\gamma$, $P^{\mu,\gamma}$ a.s. {if there exists a set of measurement sequences $A \subset \mathcal{Y}^{\mathbb{Z}_{+}}$ with $P^{\mu,\gamma}$ probability 1 such that for any sequence in $A$;} for any measure $\nu$ with $\mu \ll \nu$ we have
$
\lim_{n \to \infty} \|\pi_{n}^{\mu,\gamma}-\pi_{n}^{\nu,\gamma}\|_{TV}=0~~P^{\mu,\gamma}~a.s.
$.


\item[(iv)]  The filter is said to be \textit{universally} stable in one of the above notions if the notion holds with respect to every admissible policy $\gamma \in \Gamma$.

\end{itemize}
\end{definition}

One of the main differences between control-free and controlled partially observed Markov chains is that the filter is always Markovian under the former, whereas under a controlled model the filter process may not be Markovian since the control policy may depend on past measurements in an arbitrary (measurable) fashion. This complicates the dependency structure, and therefore results from the control-free case do not directly apply to the controlled setup. 

Recall (\ref{Dob_def}) and let us define $ \tilde{\delta}(\mathcal{T}):=\inf_{u \in \mathbb{U}} \delta(\mathcal{T}(\cdot|\cdot,u)). $
\begin{theorem} \cite[Theorem 3.3]{MYDobrushin2020} \label{curtis_result}
Assume that for $\mu,\nu \in \mathcal{P}(\mathbb{X})$, we have $\mu\ll\nu$. Then we have
\begin{align}\label{boundL_t1}
E^{\mu,\gamma}\left[\|\pi_{n+1}^{\mu,\gamma}-\pi_{n+1}^{\nu,\gamma}\|_{TV}\right]\leq \alpha E^{\mu,\gamma}\left[\|\pi_{n}^{\mu,\gamma}-\pi_{n}^{\nu,\gamma}\|_{TV}\right].
\end{align}
where $\alpha:=(1-\tilde{\delta}(\mathcal{T}))(2-\delta(Q))$.
\end{theorem}

%

If $\alpha < 1$, by applying the Borel-Cantelli lemma and Markov's inequality, we have that exponential stability in expectation implies the same result in an almost sure sense as well; see \cite[Remark 3.10]{MYDobrushin2020}. This also establishes that the rate of convergence is uniform over all priors $\nu$ as long as $\mu \ll \nu$.

A further method, and one which leads to complementary conditions given the above, for filter stability is via the Hilbert projective metric.

\begin{definition}
    Two non-negative measures $\mu, \nu$ on $(\mathbb{X},\mathcal{B}(\mathbb{X}))$ are comparable, if there exist positive constants $0<a \leq b$, such that
    $$
    a \nu(A) \leq \mu(A) \leq b \nu(A)
    $$
    for any Borel subset $A \subset \mathbb{X}$.
\end{definition}

\begin{definition}[Mixing kernel]\label{mixingKernelDef}
    The non-negative kernel $K$ defined on $\mathbb{X}$ is mixing, if there exists a constant $0<\varepsilon \leq 1$, and a non-negative measure $\lambda$ on $\mathbb{X}$, such that
    $$
    \varepsilon \lambda(A) \leq K(x, A) \leq \frac{1}{\varepsilon} \lambda(A)
    $$
    for any $x \in \mathbb{X}$, and any Borel subset $A \subset \mathbb{X}$.
    \end{definition}

\begin{definition}(Hilbert metric).  Let $\mu, \nu$ be two non-negative finite measures. We define the Hilbert metric on such measures as
    \begin{equation}
    h(\mu, \nu)= \begin{cases}\log \left(\frac{\sup _{A \mid \nu(A)>0} \frac{\mu(A)}{\nu(A)}}{\inf _{A \mid \nu(A)>0} \frac{\mu(A)}{\nu(A)}}\right) & \text { if } \mu, \nu \text { are comparable } \\ 0 & \text { if } \mu=\nu=0 \\ \infty & \text { else }\end{cases}
    \end{equation}
\end{definition}

Note that $h(a\mu, b\nu) = h(\mu, \nu)$ for any positive scalars $a, b$. Therefore, the Hilbert metric is a useful metric for nonlinear filters since it is invariant under normalization, and the following lemma demonstrates that it bounds the total-variation distance.

\begin{lemma}\label{h-TV}\cite[Lemma 3.4]{le2004stability}
    Let $\mu, \nu$ be two non-negative finite measures,
    \begin{enumerate}
    \item[i.] $
    \left\|\mu-\nu\right\|_{TV} \leq \frac{2}{\log 3} h\left(\mu, \nu\right) .
    $
    \item[ii.] 
    If the nonnegative kernel $K$ is a mixing kernel (see Definition \ref{mixingKernelDef}) with constant $\epsilon$, then 
    $
    h\left(K \mu, K \nu\right) \leq \frac{1}{\varepsilon^2}\left\|\mu-\nu\right\|_{TV}.
    $
    \end{enumerate}
\end{lemma}

\begin{lemma}[\cite{le2004stability}, Lemma 3.8]\label{Birkoff} 
    (Birkhoff contraction coefficient). 
    The nonnegative linear operator $\tau$ on $\mathcal{M}^{+}(\mathbb{X})$ (positive measures on $\mathbb{X}$) 
    associated with a nonnegative kernel $K$ defined on $\mathbb{X}$
    $$
    \tau(K):=\sup _{0<h\left(\mu, \nu\right)<\infty} \frac{h\left(K \mu, K \nu\right)}{h\left(\mu, \nu\right)}=\tanh \left[\frac{1}{4} H(K)\right]
    $$
    where
    $$
    H(K):=\sup _{\mu, \nu} h\left(K \mu, K \nu\right)
    $$
    is over nonnegative measures, is a contraction (called the Birkhoff contraction coefficient) , is a contraction under the Hilbert metric if $H(K)<\infty$ (which implies $\tau(K)<1$).
\end{lemma}

Another filter stability result which will also be useful in numerical methods for POMDPs to be considered later is via the following {\it stochastic non-linear observability} definition. 


\begin{definition}\label{one_step_observability}[Stochastic Observability for Non-Linear Systems]\label{nonLObs}\cite{MYRobustControlledFS}
A POMDP is called one step observable (universal in admissible control policies) if for every $f \in C_{b}(\mathbb{X})$ and every $\epsilon>0$ there exists a measurable and bounded function $g$ such that $\|f(\cdot)-\int_{\mathbb{Y}}g(y)Q(dy|\cdot)\|_{\infty}<\epsilon$.
\end{definition}

\begin{theorem}\label{weak_merging_pred}\cite{MYRobustControlledFS}
Assume that $\mu \ll \nu$ and that the POMDP is one step observable. Then the predictor is universally stable weakly a.s. .
\end{theorem}

The observability notion defined above only results in stability of the predictor in the weak sense $P^{\mu,\gamma}$ almost surely. However, these can be extended to filter stability and under further criteria, see \cite{MYRobustControlledFS,mcdonald2018stability}. 

We now present an example for observability. 
\begin{example}\cite{mcdonald2018stability}\label{exampleCurtisFinite} 
Consider a finite setup $\mathbb{X}=\{a_{1},\cdots, a_{n}\}$ and let the noise space be $\mathbb{V}=\{b_{1},\cdots,b_{m}\}$. Now, assume $y=h(x,v)$ has $K$ distinct outputs, where $1\leq K \leq (n)(m)$ and $\mathbb{Y}=\{c_{1},\cdots,c_{K}\}$. We note that for such a setup, there is already a sufficient and necessary condition for filter stability provided in \cite[Theorem V.2]{van2010nonlinear} (see also \cite{van2009observability}). For each $x$, $h_{x}$ can be viewed as a partition of $\mathbb{V}$, assigning each $b_{i} \in \mathbb{X}$ to an output level $c_{j}\in \mathbb{Y}$. We can track this by the matrix $H_{x}(i,j)=1$ if $h_{x}(b_{i})=y_{j}$ and zero else. Let $Q$ be the $1 \times m$ vector representing the probability measure of the noise. Let $g(c_{i})=\alpha_{i}$, with $\alpha^\intercal =\left[\alpha_1 ,\alpha_2, \dots , \alpha_K \right]$ and $\int_{\mathbb{V}} g(h(x,v))Q(dv) =: (QH_{x}) \alpha
$. Any function $f(x)$ can be expressed as a $n \times 1$ vector and hence the question reduces to finding a vector $\alpha$ so that $f =  QH \alpha$, and the system is one step observable if and only if the matrix
$
A\equiv\begin{bmatrix}
QH_{a_{1}}\\
\vdots\\
QH_{a_{n}}
\end{bmatrix}
$
is rank $n$. 
%
\end{example}

Further examples for measurement channels satisfying Definition \ref{one_step_observability} have been reported in \cite[Section 3]{mcdonald2018stability}.

Applications of these will be discussed in the context of numerical methods for POMDPs. Filter stability is also related to robustness of optimal costs to incorrect initializations for controlled models \cite{MYRobustControlledFS}.

\section{Existence of Optimal Policies: Discounted Cost and Average Cost}

\subsection{Discounted Cost}

\begin{theorem}
If the cost function $c: \mathbb{X} \times \mathbb{U} \to \mathbb{R}$ is continuous and bounded, and $\mathbb{U}$ is compact, under Theorems \ref{TV_channel_thm} or \ref{TV_kernel_thm}, for any $\beta \in (0,1)$, there exists an optimal solution to the discounted cost optimality problem with a continuous and bounded value function. Furthermore, under Assumption \ref{main_assumption}, with $K_2=\frac{\alpha D (3-2\delta(Q))}{2}$, if $\beta K_2 < 1$ the value function is Lipschitz continuous.
\end{theorem}

\begin{proof} An application of the dominated convergence theorem implies that $\tilde{c}(\pi,u)$ (\ref{weak:eq8}) is also continuous and bounded. If the action set is compact, then under Theorems \ref{TV_channel_thm} or \ref{TV_kernel_thm}, which imply that $\eta$ is weakly continuous, we have that the measurable selection conditions (see e.g. \cite{HernandezLermaMCP}) apply, and solutions to the Bellman or discounted cost optimality equations exist, and accordingly an optimal control policy exists. For the second result, \cite[Theorem 4.37]{SaLiYuSpringer} leads to Lipschitz regularity under the Wasserstein continuity condition on the kernel. 
\end{proof}

\subsection{Average Cost}

The average cost is a significantly more challenging problem as the typical contraction conditions via minorization is too demanding for $\eta$. An alternative approach is based on the Section \ref{wass_cont}.
The average cost optimality equation (ACOE) plays a crucial role for the analysis and the existence results of MDPs under the infinite horizon average cost optimality criteria. The triplet $(h,\rho^*,\gamma^*)$, where $h,\gamma:\cal{P}(\mathbb{X})\to \mathds{R}$ are measurable functions and $\rho*\in\mathds{R}$ is a constant, forms the ACOE if 
\begin{align}\label{acoe}
h(z)+\rho^*&=\inf_{u\in\mathds{U}}\left\{\tilde{c}(z,u) + \int h(z_1)\eta(dz_1|z,u)\right\}\nonumber\\
&=\tilde{c}(z,\gamma^*(z)) + \int h(z_1)\eta(dz_1|z,\gamma^*(z))
\end{align}
for all $z\in\cal{P}(\mathbb{X})$. It is well known that (see e.g. \cite[Theorem 5.2.4]{HernandezLermaMCP}) if (\ref{acoe}) is satisfied with the triplet $(h,\rho^*,\gamma^*)$, and furthermore if $h$ satisfies
\begin{align*}
\sup_{\gamma\in\Gamma}\lim_{t \to \infty}\frac{E_z^\gamma[h(Z_t)] }{t}=0, \quad \forall z\in\cal{P}(\mathbb{X})
\end{align*}
then $\gamma^*$ is an optimal policy for the POMDP under the infinite horizon average cost optimality criteria, and 
\begin{align*}
J^*(z)=\inf_{\gamma\in\Gamma}J(z,\gamma)=\rho^* \quad \forall z\in \cal{P}(\mathbb{X}).
\end{align*}
\begin{theorem}\label{mainEmre}
\begin{itemize}
\item[(i)]  \cite{demirci2023average} Under Assumption \ref{main_assumption}, with $K_2=\frac{\alpha D (3-2\delta(Q))}{2} < 1$, 
    a solution to the average cost optimality 
    equation (ACOE) exists. 
    This leads to the existence of an optimal 
    control policy, and optimal cost is constant for 
    every initial state.
  \item[(ii)]  \cite[Theorem 3]{anotherLookPOMDPs} If the cost function $c: \mathbb{X} \times \mathbb{U} \to \mathbb{R}$ is continuous and bounded, and $\mathbb{U}$ is compact, under weak Feller regularity of $\eta$ (e.g., under either Theorem \ref{TV_channel_thm} or \ref{TV_kernel_thm}), there exists an optimal policy \footnote{Here, the optimality result may only hold for a restrictive class of initial conditions or initializations, unlike part (i).}  
  \end{itemize}
\end{theorem}

\begin{proof} (i) follows from a vanishing discount method \cite{demirci2023average}. (ii) follows from the convex analytic method building on \cite{Borkar2}. 
\end{proof}


\section{Approximations: Discounted Cost}

\subsection{State and Action Space Quantization}\label{finite_belief}


By combining the approximation results in \cite{SYLTAC2017POMDP,SaLiYuSpringer}, together with the weak Feller continuity results presented earlier, we can conclude that the numerical methods for weakly continuous fully observed MDPs can also be applied to POMDPs under the conditions reported in Theorems \ref{TV_channel_thm} and \ref{TV_kernel_thm}. This has explicitly been demonstrated in  \cite{SYLTAC2017POMDP}, where also methods for quantizing probability measures have been studied in \cite[Section 5]{SYLTAC2017POMDP}. Notably, one can first quantize the action space with arbitrarily small loss (see \cite{saldi2014near}\cite[Theorem 3.16]{SaLiYuSpringer} for discounted cost and \cite{saldi2014near},\cite[Theorem 3.22]{SaLiYuSpringer} for average cost) and then approximate the probability measures, e.g. under the $W_1$ metric, to obtain a finite model. In the following, we follow the approach and results from  \cite{KSYContQLearning} applied to belief-MDPs.


To construct a finite near-optimal MDP model, we begin by 
quantizing the belief states. We select disjoint 
sets $\left\{Z_i\right\}_{i=1}^M$ such that 
$\bigcup_i Z_i={\cal P}(\mathbb{X})$, and each 
$Z_i$ is disjoint from $Z_j$ 
for any $i \neq j$.
For each set, we choose a representative state,
denoted as $z_i \in Z_i$.
This results in a finite state space for our 
model, represented by
$\bar{Z}:=\left\{z_1, \ldots, z_M\right\}$.
The quantization function maps the original state space
${\cal P}(\mathbb{X})$ to this finite set $\bar{Z}$ as follows:
$$
q(z)=z_i \quad \text { if } z \in Z_i .
$$
To define the approximate cost function, we select a 
weight measure $\pi^* \in {\cal P}({\cal P}(\mathbb{X}))$ 
over ${\cal P}(\mathbb{X})$ such that 
$\pi^*\left(Z_i\right)>0$ for all 
$Z_i$.
Under Assumption \ref{main_assumption},
we know that ${\cal P}(\mathbb{X})$ is compact under $W_1$ metric.
We then define normalized measures 
for each quantization bin $Z_i$
using the weight measure as:
$$
\hat{\pi}_{z_i}^*(A):=\frac{\pi^*(A)}{\pi^*\left(Z_i\right)}, 
\quad \forall A \subset Z_i, \quad \forall i \in\{1, \ldots, M\}.
$$
This normalized measure, $\hat{\pi}_{z_i}^*$,
is specific to the set 
$Z_i$ containing $z_i$.

Next, we define the stage-wise cost and the 
transition kernel for the MDP with the 
finite state space $\bar{Z}$ 
using these normalized weight measures.
For any $z_i, z_j \in \bar{Z}$ and 
$u \in \mathbb{U}$, the stage-wise cost function
and the transition kernel are:
$$
\begin{aligned}
c^*\left(z_i, u\right) & 
=\int_{Z_i} \tilde{c}(z, u) \hat{\pi}_{z_i}^*(d z), \\
\eta^*\left(z_j \mid z_i, u\right) & 
=\int_{Z_i} \eta\left(Z_j \mid z, u\right) 
\hat{\pi}_{z_i}^*(d z) .
\end{aligned}
$$

After establishing the finite state space 
$\bar{Z}$, the cost function $c^*$ and the transition kernel 
$\eta^*$, we introduce the discounted optimal 
value function for this finite model, denoted as 
$\hat{J}_\beta: \bar{Z} \rightarrow \mathbb{R}$.
We extend this function to the entire original state space 
${\cal P}(\mathbb{X})$  by keeping it constant within the 
quantization bins. Therefore, for any 
$z \in Z_i$, we define:
$$
\hat{J}_\beta(z):=\hat{J}_\beta(z_i) .
$$
We also define the maximum loss function among 
the quantization bins as:
\begin{align}\label{L_max}
\bar{L}:=\max _{i=1, \ldots, M} \sup _{z, z^{\prime} \in Z_i}W_1(z,z^{\prime}).
\end{align}
\begin{assumption} \label{quantized_as} [\cite{KSYContQLearning} Assumption 4]
 \noindent 
    \begin{enumerate}
        \item ${\cal P}(\mathbb{X})$ is compact (under $W_1$ metric).
        \item There exists $\alpha_c>0$ such that 
        $\left|\tilde{c}(z, u)-\tilde{c}\left(z^{\prime},
         u\right)\right| \leq \alpha_c d(z,z^{\prime})$ 
         for all $z, z^{\prime} \in {\cal P}(\mathbb{X})$ and for all $u \in \mathbb{U}$. It suffices that $\left|c(x,u)-c(x',u)\right|\leq \alpha_c|x-x'|$.
        \item There exists $\alpha_\eta>0$ such that 
        $W_1\left(\eta(\cdot \mid z, u), 
        \eta\left(\cdot \mid z^{\prime}, 
        u\right)\right) \leq \alpha_\eta d(z,z^{\prime})$ 
        for all $z, z^{\prime} \in {\cal P}(\mathbb{X})$ 
        and for all $u \in \mathbb{U}$.
    \end{enumerate}

\end{assumption}


\begin{theorem}\label{theoremBoundAvg} \cite[Theorem 6]{KSYContQLearning}
Under Assumption \ref{quantized_as}, we have
$$
\sup _{z \in {\cal P}(\mathbb{X})}\left|J_\beta\left(z, 
\hat{\gamma}\right)-J_\beta^*\left(z\right)\right| 
\leq \frac{2 \alpha_c}{(1-\beta)^2\left(1-\beta 
\alpha_\eta\right)} \bar{L} ,
$$
where $\bar{L}$ is defined in (\ref{L_max}) and 
$\hat{\gamma}$ denotes the optimal policy of 
the finite-state approximate model extended to 
the state space ${\cal P}(\mathbb{X})$ via the quantization function $q$.
\end{theorem}
A similar result is presented in the 
\cite{SaLiYuSpringer}, Theorem 4.38, offering
a slightly weaker bound.

Under Assumption \ref{main_assumption}, the belief MDP satisfies
 Assumption \ref{quantized_as} because \( {\cal P}(\mathbb{X}) \) is compact 
 under \( W_1 \) metric. It also follows that we have \( |\tilde{c}(z, u)-\tilde{c}(z', u)| \leq K_1 W_1(z,z') \)
for all \( z, z' \in {\cal P}(\mathbb{X}) \) and for all \( u \in \mathbb{U} \). 
Theorem \ref{ergodicity} implies \( W_1(\eta(\cdot \mid z, u), 
\eta(\cdot \mid z', u)) \leq K_2 W_1(z,z') \) for all 
\( z, z' \in {\cal P}(\mathbb{X}) \) and for all \( u \in \mathbb{U} \). 
Thus, for belief MDP, quantization provides the following bound:
\[ \sup _{z \in {\cal P}(\mathbb{X})}\left|J_\beta\left(z, \hat{\gamma}\right)-J_\beta^*\left(z\right)\right| \leq \frac{2 K_1}{(1-\beta)^2(1-\beta K_2)} \bar{L} . \]
Furthermore, quantized model gives near-optimal policy of the 
original belief MDP model as $\bar{L} \rightarrow 0$. 

If we only have weak Feller regularity of $\eta$ (e.g., under either Theorem \ref{TV_channel_thm} or \ref{TV_kernel_thm}), a similar approximation result holds though only by asymptotic convergence of the approximation error to zero as the diameters of the quantization bins converge to zero; see  \cite[Theorem 3]{SYLTAC2017POMDP} as an application of \cite{SaYuLi15c} and \cite[Theorem 4.27]{SaLiYuSpringer}.



In the following, an alternative approach is presented. 

\subsection{An Alternative Finite Window Belief-MDP Reduction and Its Approximation}\label{alt_section}

In this section we construct an alternative fully observed MDP reduction with the condition that the controller has observed at least $N$ information variables, using the predictor from $N$ stages earlier and the most recent $N$ information variables (that is, measurements and actions). This new construction allows us to highlight the most recent information variables and {\it compress} the information coming from the past history via the predictor as a probability measure valued variable.

Inspired from filter stability, consider the following: For any time step $t\geq N$ and for a fixed observation realization sequence $y_{[0,t]}$ and control action sequence $u_{[0,t-1]}$, the state process can be viewed as
\begin{align*}
&P^{\mu}(x_t\in \cdot|y_{[0,t]},u_{[0,t-1]})=P^{\pi_{{t-N}_-}}(X_t\in\cdot|y_{[t-N,t]},u_{[t-N,t-1]})
\end{align*}
where 
\begin{align*}
\pi_{{t-N}_-}(\cdot)=P^{\mu}(x_{t-N}\in \cdot|y_{[0,t-N-1]},u_{[0,t-N-1]}).
\end{align*}

That is, we can view the state as the Bayesian update of $\pi_{t-N_-}$, the predictor at time $t-N$, using the observations $y_{t-N},\dots,y_{t}$. Notice that with this representation only the most recent $N$ observation realizations are used for the update and the past information of the observations is embedded in $\pi_{t-N_-}$. 

We define the new state variable as the triple $(\pi_{t-N}^-,y_{[0,t-N-1]},u_{[0,t-N-1]})$.
We place the product metric on this new space: weak convergence on the belief and usual metric on the measurements and actions.



The idea is to quantize the new state as follows: collapse all $\pi$ to a fixed state $\hat{\pi}$, define an approximate finite MDP and establish performance bounds utilizing filter stability.

In the following, we will assume that $\mathbb{X}$ is $\mathbb{R}^n$ for some $n$ and that $\mathbb{U}, \mathbb{Y}$ are finite sets.

Define the quantization map $F$, such that for $(\pi,y_{[0,N]},u_{[0,N-1]})$
\begin{align*}
 F(\pi,y_{[0,N]},u_{[0,N-1]})=(\hat{\pi},y_{[0,N]},u_{[0,N-1]}).
\end{align*}

Using the map $F$ and the finite set $\mathcal{Z}^N$, one can define a finite belief MDP, and construct a policy for this finite model, by extending it, we can use the policy, say $\tilde{\phi}^N$  for the original model.


The cost function for the approximate model is
\begin{align*}
\hat{c}(\hat{z}^N_t,u_t)&=\hat{c}(\hat{\pi},I_t^N,u_t):=\tilde{c}(\phi(\hat{\pi},I_t^N),u_t)\\
&=\int_\mathbb{X}c(x_t,u_t)P^{\hat{\pi}}(dx_t|y_t,\dots,y_{t-N},u_{t-1},\dots,u_{t-N}).
\end{align*}
We define the controlled transition model for the approximate model by 
\begin{align}\label{eta_N}
&\hat{\eta}^N(\hat{z}_{t+1}^N|\hat{z}_t^N,u_t) \nonumber \\
&=\hat{\eta}^N(\hat{\pi},I_{t+1}^N|\hat{\pi},I_t^N,u_t):=\hat{\eta}\bigg({\cal P}(\mathbb{X}),I_{t+1}^N|\hat{\pi},I_t^N,u_t\bigg) 
\end{align}

We will write $\mathcal{Z}^N_{\hat{\pi}}$ to make the dependence on $\hat{\pi}$ and $N$ more explicit.



We denote the optimal value function for the approximate model by $J_\beta^N$, and the optimal policy for the approximate model by $\phi^N$. 

We investigate the following approximation error terms:
\begin{align*}
|J^N_\beta(\hat{z})-J^*_\beta(\hat{z})|, J_\beta(\hat{z},{\phi}^N) -J^*_\beta(\hat{z}).
\end{align*}
The first one is the difference between the optimal value function of the original model and that for the approximate model. The second term is the performance loss due to the policy calculated for the approximate model using finite memory being applied to the true model.

Building on \cite{kara2020robustness,KaraYuksel2021Chapter}, we can show that the loss is related to the term:
\begin{align}\label{loss_constant}
&L^N_t:= \sup_{\hat{\gamma}\in\hat{\Gamma}}E_{\pi_0^-}^{\hat{\gamma}}\bigg[\|P^{\pi_t^-}(X_{t+N}\in\cdot|Y_{[t,t+N]},U_{[t,t+N-1]}) \nonumber \\
& \quad \quad \qquad -P^{\hat{\pi}}(X_{t+N}\in\cdot|Y_{[t,t+N]},U_{[t,t+N-1]})\|_{TV} \bigg].
\end{align}

Let us elaborate on this term further. Consider the measurable policy space with respect to the new state space $\hat{\mathcal{Z}}={\cal P}(\mathds{X})\times \mathds{Y}^{N+1}\times \mathds{U}^{N}$ by $\hat{\Gamma}$. That is, a policy $\hat{\gamma}\in\hat{\Gamma}$ is a sequence of control functions $\{\hat{\gamma}_t\}$ such
that $\hat{\gamma}_t$ is measurable with respect to the $\sigma$-algebra
generated by the information variables $\{\hat{z}_0,\dots,\hat{z}_t\}$. $L^N_t$ above is then the expected bound on the total variation distance between the posterior distributions of $X_{t+N}$ conditioned on the same observation and control action variables $Y_{[t,t+N]},U_{[t,t+N-1]}$ when the prior distributions of $X_{t}$ are given by $\pi_t^-$ and $\pi^*$. The expectation is with respect to the random realizations of $\pi_t^-$ and $Y_{[t,t+N]},U_{[t,t+N-1]}$ under the true dynamics of the system when the prior distribution of $x_0$ is given by $\pi_0^-$.  This constant represents the bound on the distance of two processes with different starting points when they are updated with the same observation and action processes under the given policy. This term is directly related to filter stability, with bounds to be presented in the following.

\begin{theorem}\label{cont_bound}\cite{kara2021convergence} [Continuity of Value Functions]
For $\hat{z}_0=(\pi_0^-,I_0^N)$, if a policy $\hat{\gamma}$ acts on the first $N$ steps of the process which produces $I_0^N$, we then have
\begin{align*}
E_{\pi_0^-}^{\hat{\gamma}}\left[\left|\tilde{J}^N_\beta(\hat{z}_0)-J^*_\beta(\hat{z}_0)\right||I_0^N\right]\leq  \frac{\|c\|_\infty }{(1-\beta)}\sum_{t=0}^\infty\beta^tL^N_t
\end{align*}
\end{theorem}

\begin{theorem}\label{robust_bound}\cite{kara2021convergence} [Near Optimality of Approximate Finite Window Model Solution applied to Actual Model]
For $\hat{z}_0=(\pi_0^-,I_0^N)$, with a policy $\hat{\gamma}$ acting on the first $N$ steps,
\begin{align}\label{Kara_yuksel}
E_{\pi_0^-}^{\hat{\gamma}}\left[\left|J_\beta(\hat{z}_0,\tilde{\phi}^N) -J^*_\beta(\hat{z}_0)\right||I_0^N\right]\leq  \frac{2\|c\|_\infty }{(1-\beta)}\sum_{t=0}^\infty\beta^tL^N_t.
\end{align}
\end{theorem}


Via a somewhat different, and more direct, derivation,\cite[Section 4.2 and Theorem 17]{kara2020near} presented the following alternative condition involving sample path-wise uniform filter stability term
 \begin{align}\label{TVUnifB}
&\bar{L}^N_{TV}:=\sup_{z\in \mathcal{P}(\mathds{X})}\sup_{y_{[0,N]},u_{[0,N-1]}} \nonumber \\
& \left\|P^{z}(\cdot|y_{[0,N]},u_{[0,N-1]})-P^{z^*}(\cdot|y_{[0,N]},u_{[0,N-1]})\right\|_{TV},
\end{align}
to show the following {\it uniform} error bound:
\begin{align}\label{jmlrboundF}
\sup_z\left|J_\beta(z,\gamma_N)-J^*_\beta(z)\right|\leq \frac{2(1+(\alpha_{\mathcal{Z}}-1)\beta)}{(1-\beta)^3(1-\alpha_{\mathcal{Z}}\beta)}\|c\|_\infty \bar{L}^N_{TV}
\end{align}
for all $\beta \in (0,1)$ under a contraction condition, for some constant $\alpha_{\mathcal{Z}}$ defined in \cite{kara2020near}. Additionally, \cite[Theorem 9]{kara2020near} provided conditions where the error is in the bounded-Lipschitz metric (which is equivalent to the Wasserstein-1 metric when the state space $\mathbb{X}$ is compact), however these were only applicable for a restrictive subset of the discount parameter $\beta$. On the other hand, the bound in (\ref{Kara_yuksel}) is in expectation whereas the bound in (\ref{jmlrboundF}) is uniform, and thus the results are complementary.

\subsubsection{Explicit filter stability bounds on expected filter error $L^N_t$ and Sample Path Filter Error $\bar{L}^N_{TV}$}

 \cite{kara2021convergence} shows that the term $L^N_t$ can be bounded via Theorem \ref{curtis_result}: Recall that this states that
\begin{align}\label{boundL_t}
E^{\mu,\gamma}\left[\|\pi_{n}^{\mu,\gamma}-\pi_{n}^{\nu,\gamma}\|_{TV}\right]\leq 2\alpha^n.
\end{align}
which holds uniformly for all $\mu\ll\nu$ where $\alpha:=(1-\tilde{\delta}(\mathcal{T}))(2-\delta(Q))$, and $\delta(\cdot)$ denotes the Dobrushin coefficient of its argument (stochastic kernel). Since $\tilde{\delta}(\mathcal{T})$ is a uniform Dobrushin coefficient over all control actions, the above bound is valid under any control policy. Thus we have that $L^N_t \leq 2\alpha^N$. 

As a complementary condition, via the Birkhoff-Hopf theorem, a controlled version of a contraction via the Hilbert metric \cite{le2004stability} can be utilized \cite{demirciRefined2023}:

Recall that 
\[F(z, y, u)(\cdot)=\operatorname{Pr}\left\{X_{k+1} \in \cdot \mid Z_k=z, Y_{k+1}=y, U_k=u\right\}\]
\begin{assumption}\label{mixing_kernel_con}
    \begin{enumerate}
    \item $Q(y|x)\geq \epsilon > 0$ for every $x\in \mathbb{X}$ and $y\in \mathbb{Y}$.
    \item The transition kernel $\mathcal{T}(.|.,u)$ is a mixing kernel (see Definition \ref{mixingKernelDef}) for every $u\in\mathbb{U}$.
\end{enumerate} 
\end{assumption}

\begin{lemma}\label{clm}\cite{demirci2023geometric}
    Under Assumption \ref{mixing_kernel_con}, 
    there exists a constant $r < 1$ such that
    \begin{align}
        h(F(\mu, y,u), F(\nu, y,u))\leq r h(\mu, \nu)
    \end{align}
    for every comparable $\mu,\nu\in {\cal P}(\mathbb{X})$ and for every
    $u\in \mathbb{U}$ and $y\in \mathbb{Y}$. Here $r=\frac{1-\epsilon_{u}^2\epsilon }{1+\epsilon_{u}^2\epsilon},$ $\epsilon_{u}$ is the mixing constant of the kernel $\mathcal{T}(.|.,u)$.
\end{lemma}

\begin{theorem}\label{d}\cite{demirciRefined2023}
    Under Assumption \ref{mixing_kernel_con}, 
    there exists a constant $r<1$ and $K$ such that 
    \begin{align}
        \bar{L}_{TV}^N \leq r^{N-1} K.
    \end{align}
    Here, $K=\frac{2}{\log 3} \sup h(Z_1, Z_1^*)$ and $r=\sup_{u\in \mathbb{U}} \frac{1-\epsilon_{u}^2\epsilon }{1+\epsilon_{u}^2\epsilon}$.
\end{theorem}

\begin{corollary}\label{cor_kara_y}\cite{demirciRefined2023}
 Under Assumption \ref{mixing_kernel_con}, 
    there exists a constant $r<1$ and $K$ such that 
    \begin{align}
    E_{z_0^{-}}^{\hat{\gamma}}\left[\left|\tilde{J}_\beta^N\left(\hat{z}_0, \tilde{\phi}^N\right)
    -J_\beta^*\left(\hat{z}_0\right)\right| \mid I_0^N\right] 
    \leq \frac{2 \|c\|_{\infty}}{(1-\beta)^2} r^{N-1} K.
    \end{align}
        Here, $K=\frac{2}{\log 3} \sup h(Z_1, Z_1^*)$ and $r=\sup_{u\in \mathbb{U}} \frac{1-\epsilon_{u}^2\epsilon }{1+\epsilon_{u}^2\epsilon}$.
\end{corollary}

\begin{remark}
Among recent results, \cite{cayci2022} provides an upper error bound for finite window policies, under a persistence of excitation of the optimal policy and minorization-majorization assumptions, \cite{cayci2022} demonstrates that, as $N$ increases, the error term converges to zero geometrically. Unlike Assumption \ref{mixing_kernel_con}, a persistence of excitation of the optimal policy requires that the optimal policy must be strictly non-deterministic. We note also that the state space in our setup is not necessarily finite. 
\end{remark}

Implementing the above is still tedious, though now numerically possible. Can reinforcement learning be feasible? Can we view the finite history as an approximate {\it state} to run a learning algorithm? Would such an algorithm convergence, and what would such a convergence operationally mean? We address these questions, building on \cite{karayukselNonMarkovian} and \cite{kara2021convergence}, in the following section.

\subsection{Robustness to Incorrect Models and Priors}

To complete our analysis on existence, regularity, and approximations, and before proceeding with reinforcement learning, we also review recent results on robustness to incorrect priors and models.

Suppose that we represent the cost of the model given in (\ref{updateEq1})-(\ref{updateEq2}) and cost criteria (\ref{expCost}-\ref{expDiscCost}), so that the dependence on the prior $\mu$, transition kernel ${\cal T}$, measurement channel $Q$, and the stage-wise cost $c$ is explicitly given as follows:
\begin{eqnarray}\label{expCostR}
J_{\infty}(c,\mu,{\cal T}, Q,\gamma) := \limsup_{N \to \infty} {1 \over N} E^{\gamma}_{\mu}[\sum_{k=0}^{N-1} c(X_k, U_k)],
\end{eqnarray}
with the infimum
\[J^*_{\infty}(c,\mu,{\cal T}, Q) = \inf_{\gamma \in \Gamma} J_{\infty}(c,\mu,{\cal T}, Q, \gamma)\]
or the discounted cost criterion (for some $\beta \in (0,1)$
\begin{eqnarray}\label{expDiscCostR}
J_{\beta}(c,\mu,{\cal T}, Q,\gamma) :=E^{\gamma}_{\mu}[\sum_{k=0}^{\infty} \beta^k c(X_k, U_k)], 
\end{eqnarray}
with the infimum being
\[J^*_{\beta}(c,\mu,{\cal T}, Q) = \inf_{\gamma \in \Gamma} J_{\beta}(c,\mu,{\cal T}, Q, \gamma)\]

The robustness question involves the following. 

\noindent{\bf Problem P1: Continuity of Optimal Cost under the Convergence of Models.}
Let $\{\mu_n, {\cal T}_{n},Q_n, n\in\mathbb{N}\}$ be a sequence of priors, transition kernels, and channels which converges in some sense to another model $(\mu,{\cal T}, Q)$ and $\{c_n,n\in\mathbb{N}\}$ be a sequence of stage-wise cost functions corresponding to $(\mu_n,{\cal T}_n,Q_n)$ which converge in some sense to another cost function $c$ corresponding to $(\mu,{\cal T}, Q)$. Does that imply that
\begin{align*}
  J^*_{\beta}(c_n,\mu_n,{\cal T}_n, Q_n) \to J_\beta^*(c,\mu,{\cal T}, Q)?
\end{align*}

\noindent{\bf Problem P2: Robustness to Incorrect Models.} 
 A problem of major practical importance is robustness of an optimal controller to modeling errors. Suppose that an optimal policy is constructed according to a model which is incorrect: how does the application of the control to the true model affect the system performance and does the error decrease to zero as the models become closer to each other? In particular, suppose that $\gamma_n^*$ is an optimal policy designed for $\{c_n,\mu_n, {\cal T}_{n},Q_n, n\in\mathbb{N}\}$. Is it the case that if $(c_n,\mu_n,{\cal T}_n, Q_n)$ converges in some appropriate sense to $(c,\mu,{\cal T}, Q)$, then
\[J_\beta^*(c,\mu,{\cal T}, Q,\gamma_n^*) \to J_\beta^*(c,\mu,{\cal T}, Q).\] 
The case where only $\mu_n \to \mu$ while the other parameter are fixed is referred to as {\it robustness to incorrect priors}. \\

\noindent{\bf Problem P3:  Empirical Consistency of Learned Probabilistic Models and Data-Driven Stochastic Control.}
Let $(\mathcal{T}(\cdot|x,u), Q(\cdot|x))$ be the transition and measurement kernels, which is unknown to the decision maker (DM). Suppose the DM builds a model for these kernels, $(\mathcal{T}_n(\cdot|x,u), Q_n(dy|x))$, for all possible $x\in \mathbb{X}, u \in \mathbb{U}$ by collecting training data (e.g. from an evolving system). Do we have that the optimal cost calculated under $(\mathcal{T}_n,Q_n)$ converges to the true cost (i.e., do we have that the cost obtained from applying the optimal policy for the empirical model converges to the true cost as the training length increases)?

\subsubsection{Robustness to incorrect priors}

We refer the reader to \cite[Theorem 3.2]{KYSICONPrior} and with further refinements under filter stability \cite[Theorem 3.8]{MYRobustControlledFS} for discounted cost and \cite{MYRobustControlledFS}, Theorem 3.9] for average cost. These show that the problem is robust to uncertainty in priors under total variation, and for robustness under weak convergence, total variation continuity of the channel as in \ref{TV_channel}(ii) is to be imposed. Further regularity results are present in \cite{KYSICONPrior}. Under filter stability, stronger robustness conditions are presented in \cite{MYRobustControlledFS}.

\subsubsection{Robustness to incorrect models}

We refer the reader to \cite{kara2020robustness,KaraYuksel2021Chapter,kara2022robustness} for robustness to transition kernel, cost functions (and which also apply to that in measurement channels); and to \cite{YukselOptimizationofChannels,YukselBasarBook24} for the special case of convergence of measurement channels. To present a flavour of results, we state the following.
\begin{assumption}\label{weak_assmp}
\begin{itemize}
\item[(i)] The sequence of transition kernels $\mathcal{T}_n$ satisfies the following: $\{\mathcal{T}_n(\cdot|x_n,u_n),n\in\mathbb{N}\}$ converges weakly to $\mathcal{T}(\cdot|x,u)$ for any sequence $\{x_n,u_n\}\subset\mathbb{X}\times\mathbb{U}$ and $x,u \in \mathbb{X}\times\mathbb{U}$ such that $(x_n,u_n)\to (x,u)$ (this is referred to as {\it continuous} weak convergence \cite{kara2020robustness,KaraYuksel2021Chapter,kara2022robustness}).
\item [(ii)] The stochastic kernel $\mathcal{T}(\cdot|x,u)$ is weakly continuous in $(x,u)$.
\item[(iii)] The sequence of stage-wise cost functions $c_n$ satisfies the following: $c_n(x_n,u_n)$\linebreak$\to c(x,u)$ for any sequence $\{x_n,u_n\}\subset\mathbb{X}\times\mathbb{U}$ and $x,u \in \mathbb{X}\times\mathbb{U}$ such that $(x_n,u_n)\to (x,u)$.
\item[(iv)] The stage-wise cost function $c(x,u)$ is non-negative, bounded, and continuous on $\mathbb{X} \times \mathbb{U}$.
\item [(v)] $\mathbb{U}$ is compact. 
\end{itemize}
\end{assumption}

The following hold:
\begin{itemize}
\item[(a)] Continuity and robustness do not hold in general under weak convergence of kernels.
\item[(b)] Under Assumptions \ref{weak_assmp} and \ref{TV_channel}(ii), continuity and robustness hold.
\item[(c)] Continuity and robustness do not hold in general under setwise convergence of the kernels.
\item[(d)] Continuity and robustness do not hold in general under total variation convergence of the kernels.
\item[(e)] Continuity and robustness hold under {\it continuous} total variation convergence of the kernels (i.e. if $\mathcal{T}_n(\cdot|x,u_n) \to \mathcal{T}(\cdot|x,u)$ in total variation for any $u_n\to u$ and for any $x$).
\end{itemize}
The above has direct implications on data-driven learning and empirical consistency, where empirical models are constructed via data, and empirical models converge weakly (and under the $W_1$ distance) almost surely (\cite{Dud02}, Theorem 11.4.1), but do not so under total variation unless density conditions are present \cite[Chapter 3]{DevroyeGyorfi}; see \cite{kara2020robustness,KaraYuksel2021Chapter,kara2022robustness}.

\section{Reinforcement Learning for POMDPs: Discounted Cost}

\subsection{A General Convergence Result for Asymptotically Ergodic Processes}\label{nonMarkov_sec}

We summarize the following, building on \cite{karayukselNonMarkovian}. Let $\{C_t\}_{t}$ be $\mathbb{R}$-valued, $\{S_t\}_t$ be $\mathbb{S}$-valued and $\{U_t\}_{t}$ be $\mathbb{U}$-valued three stochastic processes. Consider the following iteration defined for each $(s,u)\in\mathbb{S}\times\mathbb{U}$ pair
\begin{align}\label{iterateAlgM}
Q_{t+1}(s,u)=&\left(1-\alpha_t(s,u)\right)Q_t(s,u) \nonumber\\
&+ \alpha_t(s,u)\left(C_t + \beta V_t(S_{t+1}) \right)
\end{align}
where $V_t(s)=\min_{u\in\mathbb{U}} Q_t(s,u)$, and $\alpha_t(s,u)$ is a sequence of constants also called the learning rates.  Note also that we overwrite the notation in this section by using $Q$ for the Q values which is different than the channel kernel $Q(\cdot|x)$.

Consider the following equation
\begin{align}
Q^*(s,u)= C^*(s,u) + \beta \sum_{s_1\in \mathbb{S}}V^*(s_1)P^*(s_1|s,u)\label{Q_fixed}
\end{align}
for some functions $Q^*$, $C^*$, to be defined explicitly, and for some regular conditional probability distribution $P^*(\cdot|s,u)$, where $V^*(u):=\min_u Q^*(s,u)$.

An umbrella sufficient condition is the following:
\begin{assumption}\label{erg_assmp}
$\mathbb{S},\mathbb{U}$ are finite sets, and the joint process $(S_{t+1},S_t,U_t,C_t)$ is asymptotically ergodic in the sense that for the given initialization random variable $S_0$, for any measurable function $f$, we have that with probability one,
\begin{align*}
&\lim_{N \to \infty} \frac{1}{N}\sum_{t=0}^{N-1} f(S_{t+1},S_t,U_t,C_t) \\
& \quad \quad = \int f(s_1,s,u,c)\phi(ds_1,ds,du,dc) 
\end{align*}
for some measure $\phi$ such that the marginal on the second and third coordinates $\phi(\mathbb{S} \times B \times \mathbb{R})>0$ for any non-empty $ B  \subset \mathbb{S}\times \mathbb{U}$.
\end{assumption}

We note that although we assume that the spaces $\mathds{S,U}$ are finite, we will continue using integral and differential notation for consistency.

The above implies Assumption \ref{main_assmp}(ii)-(iii) below:

\begin{assumption}\label{main_assmp}
\begin{itemize}
\item[i.] $\alpha_t(s,u)=0$ unless $(S_t,U_t)=(s,u)$. Furthermore,
\begin{align*}
\alpha_t(s,u)=\frac{1}{1+\sum_{k=0}^t1_{\{S_k=s,U_k=u\}}}
\end{align*}
and with probability $1$,
$\sum_t \alpha_t(s,u) = \infty$.
\item[ii.] For $C_t$, we have
\begin{align*}
\frac{\sum_{k=0}^t C_k 1_{\{S_k=s,U_k=u\}}}{\sum_{k=0}^t  1_{\{S_k=s,U_k=u\}}}\to C^*(s,u),
\end{align*}
almost surely for some $C^*$.
\item[iii.] For the $S_t$ process, we have, for any function $f$,
\begin{align*}
\frac{\sum_{k=0}^t f(S_{k+1}) 1_{\{S_k=s,U_k=u\}}}{\sum_{k=0}^t  1_{\{S_k=s,U_k=u\}}}\to \int f(s_1)P^*(ds_1|s,u)
\end{align*}
almost surely for some $P^*$.
\end{itemize}
\end{assumption}

Recently, \cite{karayukselNonMarkovian} presented conditions for the convergence of the iterates above: 
\begin{theorem}\label{main_thm}\cite{karayukselNonMarkovian}
Under Assumption \ref{main_assmp}, $Q_t(s,u)\to Q^*(s,u)$ almost surely for each $(s,u)\in\mathbb{S}\times\mathbb{U}$ pair where $Q^*$ satisfies (\ref{Q_fixed}).
\end{theorem}

It turns out that (\ref{Q_fixed}) is the fixed point corresponding to an approximate MDP, with implications for POMDPs noted in the following (see also \cite{karayukselNonMarkovian}).

\subsection{Finite Window Memory POMDP with Uniform Geometric Controlled Filter Stability}
We will here assume that $\mathbb{X}$ is a compact subset of a Polish space and that $\mathbb{Y}$ and $\mathbb{U}$ are finite sets.

Suppose that the controller keeps a finite window of the most recent $N$ observation and control action variables, and perceives this as the {\it state} variable, which is in general non-Markovian. That is we take \[S_t = \{Y_{[t-N,t]},U_{[t-N,t-1]}\},\] and $C_t:=c(X_t,U_t)$.

In this case, $(S_t,X_t,U_t)$ form a controlled Markov chain, even if $(S_t,U_t)$ does not. We state the ergodicity assumption formally next.

\begin{assumption}\label{POMDP_unif_erg}
\begin{itemize}

\item[(i)] Under the exploration policy $\gamma$ and initialization, the controlled state and control action joint process $\{X_t,U_t\}$ is asymptotically ergodic in the sense that for any measurable function $f$ we have that
\begin{align*}
\lim_{N\to \infty}\frac{1}{N}\sum_{t=0}^{N-1}f(X_t,U_t)=\int f(x,u)\phi^\gamma(dx,du)
\end{align*}
for some $\phi^\gamma \in{\cal P}(\mathbb{X}\times\mathbb{U})$. Furthermore, we have that $P(Y_t = y | x) > 0$ for every $x \in \mathbb{X}$.
\item[(ii)] Assumption \ref{erg_assmp}(i) holds with $S_t = \{Y_{[t-N,t]},U_{[t-N,t-1]}\}$. 
\end{itemize}
\end{assumption}

We note that a sufficient condition for the ergodicity assumption, for every initialization of $X_0$, would be positive Harris recurrence under the exploration policy.

The question then is if the limit Q values correspond to a meaningful control problem, and how `close' this control problem to the original POMDP. \cite[Theorem 4.1]{kara2021convergence} shows that the limit Q values indeed correspond to an approximate control problem, and notes the following bound:

\begin{theorem} \cite[Theorem 4.1]{kara2021convergence}
Under Assumption \ref{POMDP_unif_erg}, the iterations in (\ref{iterateAlgM}) converge with $S_t = \{Y_{[t-N,t]},U_{[t-N,t-1]}\}$ and $C_t:=c(X_t,U_t)$. Furthermore, if we denote the policies constructed using these Q values by $\gamma^N$, and apply these finite memory policies in the original problem, we get the following error bound:
\begin{align*}
E\left[J_\beta(\pi_N^-,\mathcal{T},\gamma^N)-J_\beta^*(\pi_N^-,\mathcal{T})|I_0^N\right]\leq \frac{2\|c\|_\infty }{(1-\beta)}\sum_{t=0}^\infty\beta^t L^N_t
\end{align*}
where $I_0^N$ is the first $N$ observation and control variables, and the expectation is taken with respect to different realizations of $I_0^N$ under the initial distribution of the hidden state $\pi_0$ and the exploration policy $\gamma$. Furthermore, $\pi_N^-=P(X_N\in\cdot|I_0^N)$
where $L^N_t$ is given by (\ref{loss_constant}) with the fixed prior is the invariant measure on $x_t$ under the exploration policy $\gamma$. In particular, we assume that the control starts after observing at least $N$- time steps of history.
\end{theorem}

As noted earlier, $L^N_t$ is related to the filter stability problem, see (\ref{boundL_t}). 



\subsection{Quantized Approximations for Weak Feller POMDPs with only Asymptotic Filter Stability}

As noted earlier, any POMDP can be reduced to a completely observable Markov process (\cite{Yus76}, \cite{Rhe74}) (see (\ref{kernelFilter})), whose states are the posterior state distributions or {\it belief}s of the observer; that is, the state at time $t$ is the filter variable
\begin{align}
\pi_t(\,\cdot\,) := P\{X_{t} \in \,\cdot\, | y_0,\ldots,y_t, u_0, \ldots, u_{t-1}\} \in {\cal P}(\mathbb{X}). \nonumber
\end{align}
Recall the kernel $\eta$ (\ref{kernelFilter}) for the filter process. Now, by combining the quantized Q-learning and the weak Feller continuity results for the non-linear filter kernel (\cite{FeKaZg14} \cite{KSYWeakFellerSysCont}), we can conclude that the setup in Section \ref{nonMarkov_sec} and Section \ref{finite_belief} is applicable though with a significantly more tedious analysis involving ergodicity requirements. Additionally, one needs to quantize probability measures. Accordingly, we take $S_t = g(\pi_t)$ for some quantizer $g:{\cal P}(\mathbb{X}) \to {\cal P}(\mathbb{X})^M=:\{B_1,B_2, \cdots, B_{|{\cal P}(\mathbb{X})^M|}\}$ with $|{\cal P}(\mathbb{X})^M| < \infty$, and $C_t:=c(X_t,U_t)$. 

We state the ergodicity condition formally:

\begin{assumption}\label{belief_MDP_unif_erg}
Under the exploration policy $\gamma$ and initialization, the controlled belief state and control action joint process $\{\pi_t,U_t\}$ is asymptotically uniquely ergodic in the sense that for any measurable function $f$ we have that
\begin{align*}
\lim_{N\to \infty}\frac{1}{N}\sum_{t=0}^{N-1}f(\pi_t)=\int f(\pi)\eta^\gamma(d\pi)
\end{align*}
for some $\eta^\gamma \in{\cal P}({\cal P}(\mathbb{X})\times\mathbb{U})$ such that $\eta^\gamma(B)>0$ for any quantization bin $B \subset {\cal P}(\mathbb{X})$.
\end{assumption}

We refer to the set
\[{\cal P}_{\eta} := \{ \pi:  \pi \in B_i \subset {\cal P}(\mathbb{X}): \eta^\gamma(B_i)>0\},\]
as the trained set of states; since these sets will be visited infinitely often under the exploration policy. 

The condition that $\eta^\gamma(B)>0$ requires an analysis tailored for each problem. For example, if the quantization is performed as in \cite{kara2020near} by clustering bins based on a finite past window, then the condition is satisfied by requiring that $P(Y_t = y | x) > 0$ for every $x \in \mathbb{X}$. If the clustering is done, e.g. by quantization of the probability measures via first quantizing $\mathbb{X}$ and then quantizing the probability measures on the finite set (see \cite[Section 5]{SYLTAC2017POMDP}), then the initialization could be done according to the invariant probability measure corresponding to the hidden Markov source.

Unique ergodicity of the dynamics follows from results in the literature, such as, \cite[Theorem 2]{DiMasiStettner2005ergodicity} and \cite[Prop 2.1]{van2009uniformSPA}, which holds when the randomized control is memoryless under mild conditions on the process notably, the hidden variable is a uniquely ergodic Markov chain and the measurement structure satisfies filter stability in total variation in expectation (one can show that weak merging in expectation also suffices); we refer the reader to \cite[Figure 1]{mcdonald2018stability} for mild conditions leading to filter stability in this sense, which is related to stochastic observability in Definition \ref{one_step_observability} (see also \cite[Definition II.1]{mcdonald2018stability}). Notably, a uniform and geometric controlled filter stability is not required even though this would be sufficient. Therefore, due to the weak Feller property of controlled non-linear filters, we can apply the Q-learning algorithm to also belief-based models to arrive at near optimal control policies. Nonetheless, since positive Harris recurrence cannot typically be assumed for the filter process, the initial state may not be arbitrary. If the invariant measure under the exploration policy is the initial state, \cite[Prop 2.1]{van2009uniformSPA} implies that the time averages will converge as imposed in Assumption \ref{main_assmp}. A sufficient condition for unique ergodicity then is the following.
\begin{assumption}
Under the exploration policy $\gamma$ the hidden process $\{X_t\}$ is uniquely ergodic (with measure $\zeta$) and the measurement dynamics are so that the filter is stable in expectation under weak convergence.
\end{assumption}

\begin{theorem} 
Suppose that Assumption \ref{quantized_as} holds such that $\alpha_\eta\beta<1$. 
\begin{itemize}
\item[(a)]
Suppose that under the exploration policy and initialization, the controlled filter process satisfies Assumption \ref{belief_MDP_unif_erg} and \ref{main_assmp}(i) with $S_t=g(\pi_t)$, and $C_t=c(X_t,U_t)$. Then, the $Q_t$ iterates converge almost surely. 
\item[(b)] Let $\pi_0 \sim \kappa \ll \phi$ or $\pi_0 \in \mathrm{supp}(\phi)$ and under $\eta^{\gamma}$ the boundary sets of the bins have measure zero. Then, 
 the policy constructed using the limit $Q$ values, say $\hat{\gamma}$, applied to the true model leads to the following bound:
\begin{align*}
J_\beta^*(\pi_0,\hat{\gamma})- J_\beta^*(\pi_0)\leq  \frac{2\alpha_c}{(1-\beta)^2(1-\beta\alpha_{\eta})}\bar{L}.
\end{align*} 
where 
\[\bar{L}:=\sup_{\pi \in \mathrm{supp( \eta^{\gamma})}, \pi \in B_i: \eta^{\gamma}(B_i) > 0} W_1\left(\pi,g(\pi)\right).\]

\item[(c)] For asymptotic convergence (without a rate of convergence) to optimality as the quantization rate goes to $\infty$ (i.e., $\bar{L} \to 0$), only weak Feller property of $\eta$ is sufficient for the the algorithm to be near optimal. 
\end{itemize}
\end{theorem}
A sufficient condition for the assumptions of (a) and (b) above is that for exploration (i) $\pi_0 \sim \kappa \ll \phi$, or (ii) $\pi_0=\zeta$ where $\zeta$ is the invariant measure for the hidden state process under exploration and that under $\eta^{\gamma}$ the boundary sets of the bins have measure zero, or (iii) there exists $\pi \in {\cal P}(\mathbb{X})$ such that for all $\pi_0, P(\inf \{ k > 0: \pi_k = \pi \} < \infty) = 1$. For further discussion on the initialization for the algorithm during implementation, please see \cite[Lemma 6 and Corollary 2]{creggZeroDelayNoiseless}. 


\begin{remark}
We now present a comparison between the two approaches above: filter quantization vs. finite window based learning:
\begin{itemize}
\item[(i)] For the filter quantization, we only need unique ergodicity of the filter process under the exploration policy for which asymptotic filter stability in expectation in weak or total variation is sufficient. The running cost can start immediately without waiting for a window of measurements. On the other hand, the controller must run the filter and quantize it in each iteration while running the Q-learning algorithm; accordingly the controller must know the model. Additionally, the initialization cannot be arbitrary (e.g. the initialization for the filter may be the invariant measure under the exploration policy): As noted earlier, one needs to ensure that the set of bin-action pairs which are visited infinitely often during exploration is so that an optimal policy is learned (visited infinitely often), and when this optimal policy (learned via the convergence of $Q$-learning) is implemented, the closed-loop process always remains in this set; see \cite{karayukselNonMarkovian} and \cite[Lemma 6 and Corollary 2]{creggZeroDelayNoiseless}.
\item[(ii)] For the finite window approach, a uniform convergence of filter stability, via $L^N_t$, is needed and it does not appear that only asymptotic filter stability can suffice. On the other hand, this is a universal algorithm in that the controller does not need to know the model. Furthermore, the initialization satisfaction holds under explicit conditions; notably if the hidden process is positive Harris recurrent, the ergodicity condition holds for every initialization; both the convergence of the algorithm as well as its implementation will always be well-defined.
\end{itemize}
For each setup, however, we have explicit and testable conditions.
\end{remark}


\section{The Average Cost Case}\label{averageCostSetup}

Approximations and learning for POMDPs under the average cost criterion is significantly more challenging. In the classical MDP theory, the approaches primarily require strong ergodicity or minorization conditions which are not suitable for the belief-MDP. Several papers \cite{platzman1980optimal,fernandez1990remarks,runggaldier1994approximations,hsu2006existence} have studied the average-cost control problem under the assumption that the state space is finite; they provide reachability type conditions for the belief kernels. Reference \cite{Bor00} considers the finite model setup and \cite{borkar2004further} considers the case with finite-dimensional real-valued state spaces under several technical conditions on the controlled state process and \cite{StettnerSICON19} considers several conditions directly on the filter process leading to an equi-continuity condition on the relative discounted value functions. One could adopt techniques suitable for the average cost without needing minorization conditions, see \cite{demirci2023average}. 

{\bf Average Cost via Near Optimality of Discounted Cost Policies.} Building on \cite{demirci2023average}, consider the following two conditions: (i) There exists a solution to the average cost optimality equation as in (\ref{acoe}), and (ii) that this solution is obtained via the vanishing discount method. Under these conditions, it follows that (see \cite[Theorems 1 and 2]{creggZeroDelayNoiseless}  and \cite[Theorem 7.3.6]{yuksel2020control}) a near optimal policy for the discounted cost problem is also near optimal for the average cost problem. 

Theorem \ref{mainEmre}(i) implies these conditions simultaneously. Consequently:

\begin{itemize}
\item[(i)] Accordingly, under Assumption \ref{main_assumption}, with $K_2=\frac{\alpha D (3-2\delta(Q))}{2} < 1$, by Theorem \ref{mainEmre}(i), a learning method would be to approximate the Q-learning algorithm by its classical discounted version by taking $\beta$ sufficiently large. Therefore, the methods available for discounted cost also apply to the average cost problem for near optimality. For details, please see \cite{demirci2023average}. 
\item[(ii)] Another implication of Theorem \ref{mainEmre}(i) is that, generalizing \cite{yu2008near}, one can conclude that finite window policies are near optimal for average cost problems under controlled filter stability conditions. 
\item[(iii)] A further byproduct of this approach is the complete robustness to incorrect initializations for POMDPs in average cost problems, as reported in \cite[Corollary 3.1]{demirci2023average}, connecting \cite[Theorem 3.9]{MYRobustControlledFS} with Theorem \ref{mainEmre}(i). 
\end{itemize}


\section{Conclusion}
In this article a general review on partially observable Markov Decision Processes has been presented. The focus has been on regularity (including continuity and filter stability) and associated existence results, approximate optimality via finite approximations or finite memory policies, and a rigorous analysis on reinforcement learning to near optimality.

\section{Acknowledgements}
We gratefully acknowledge our collaborations with Naci Saldi, Yunus Emre Demirci, Curtis McDonald and Tam\'as Linder; and our many discussions with Prof. Eugene Feinberg, Vivek Borkar, Erhan Bayraktar, G\"urdal Arslan, and Aditya Mahajan.

\bibliographystyle{plain}
\bibliography{SerdarBibliography}

\end{document}